\documentclass[11pt,a4paper,reqno]{amsart}

\usepackage{amssymb, amsmath, textcomp, amsthm, amsfonts, bm}
\usepackage{bbm,dsfont}
\usepackage{color}
\usepackage[dvipsnames]{xcolor}
\numberwithin{equation}{section}
\usepackage{graphicx}

\title{Small cap decoupling inequalities: Bilinear methods}
\date{}

\def\R{\mathbb{R}}
\def\N{\mathbb{N}}
\def\C{\mathbb{C}}
\def\nint{\mathop{\diagup\kern-13.0pt\int}}
\def\Z{\mathbb{Z}}
\def\M{\mathcal{M}}

\def\beq{\begin{equation}}
\def\endeq{\end{equation}}
\def\beqq{\begin{equation*}}
\def\endeqq{\end{equation*}}
\def\bg{\begin{gathered}}
\def\eg{\end{gathered}}

\def\mc{\mathcal}

\theoremstyle{plain}
\newtheorem{thm}{Theorem}[section]
\newtheorem{prop}[thm]{Proposition}

\newtheorem{lem}[thm]{Lemma}

\newtheorem*{conj*}{Conjecture}

\newtheorem*{openproblem*}{Open Problem}

\makeatletter
\renewcommand*\env@matrix[1][*\c@MaxMatrixCols c]{%
	\hskip -\arraycolsep
	\let\@ifnextchar\new@ifnextchar
	\array{#1}}
\makeatother

\begin{document}
\author{Changkeun Oh}
	\address{Department of Mathematics, University of Wisconsin-Madison, USA}
\email{coh28@wisc.edu}
\maketitle
\begin{abstract}
	We obtain sharp small cap decoupling inequalities associated to the moment curve for certain range of exponents $p$. Our method is based on the bilinearization argument due to Bourgain \cite{Bou17} and Bourgain and Demeter \cite{BD17}. Our result generalizes theirs to all higher dimensions.
\end{abstract}

\section{Introduction}

Let $n\ge 2$ be an integer. For an interval $\Delta\subset [0, 1]$, define an extension operator 
\beqq
E_{n, \Delta} f(x)=\int_{\Delta} f(t) e(tx_1+\dots+t^n x_n)dt.
\endeqq
Here $x=(x_1, \dots, x_n)\in \R^n$, and $e(t):=e^{it}$ for a real number $t$. Let $r\ge 1$ and $\delta\in (0, 1]$. We use $B^r$ to denote a ball in $\R^n$ of radius $\delta^{-r}$. Moreover, for a ball $B\subset \R^n$ of radius $r_B$ and center $c_B$, we use $w_B$ to stand for a suitable weight essentially support on $B$:
\beqq
w_B(x):=(1+\frac{|x-c_B|}{r_B})^{-100 n}.
\endeqq 
Let $D_p(n, r,\delta)$ be the smallest real number such that the decoupling inequality
\beq\label{cancellation}
\|E_{n, [0, 1]} f\|_{L^p(w_{B^r})} \leq D_{p}(n,r,\delta) \Big(\sum_{\Delta\subset [0, 1]: l(\Delta)=\delta} \|E_{n, \Delta} f\|^p_{L^p(w_{B^r})}\Big)^{\frac1p}
\endeq
holds for every integrable function $f: [0,1] \rightarrow \C$. Here the sum on the right hand side is over all dyadic intervals $\Delta$ of the form $[a,a+\delta]$ with $a \in \delta\Z$. 

 We define the number $p_n$ by
\beqq\label{pn}
p_n := \left\{ \begin{array}{ll}
	2k(k+1)& \textrm{with $k=\frac{n}{2}$ when $n$ is even}\\[2pt]
	2(k+1)^2& \textrm{with $k=\frac{n-1}{2}$ when $n$ is odd}.
\end{array} \right.
\endeqq
Let $[x]$ be the greatest integer less than or equal to $x$.
The main theorem is the following.
\begin{thm}\label{mainthm}
Let $n \geq 3$. For every $2 \leq p \leq p_n$ and $\epsilon >0$, there exists some positive number $C_{n,p,\epsilon} < \infty$ such that
\beq\label{180511e1.3}
D_p(n,r,\delta) \leq  C_{n,p,\epsilon}
\delta^{-(\frac12-\frac1p)-\epsilon}
\endeq
 for $r=[\frac{n+1}{2}]$ and every $1 \geq \delta >0$.
\end{thm}

An inequality of the form \eqref{180511e1.3} is called a small cap decoupling inequality, see \cite{DGW19}. This is in comparison to the decoupling inequality for the moment curve proven by Bourgain, Demeter and Guth \cite{BDG16}. 
The authors there proved estimates of the form \eqref{180511e1.3} with $r=n$, that is, on a larger ball of radius $\delta^{-n}$. From this aspect, one can also call \eqref{180511e1.3} a small ball decoupling inequality. 
\\

By summing over all the balls in a larger ball, we see that 
\beq\label{180511e1.4}
D_p(n,r_1,\delta) \leq C_nD_p(n,r_2,\delta) \text{ whenever } r_2\le r_1
\endeq
for some constant $C_n>0$.
Moreover, by Fubini's theorem, it is not difficult to see that 
\beq\label{180511e1.5}
D_p(n_1,r,\delta) \leq C_{n_1,n_2}D_p(n_2,r,\delta) \text{ whenever } n_2\le n_1
\endeq
for some constant $C_{n_1,n_2}>0$.
Hence, combining \eqref{180511e1.3}--\eqref{180511e1.5}, one would be able to obtain upper bounds on $D_p(n,r,\delta)$ for other pairs of $n$ and $r$.\\

The exponent of $\delta$ in the inequality \eqref{180511e1.3} is sharp up to $\delta^{-\epsilon}$ losses. However, we do not know whether the range of $p$ is optimal.
It is desirable to prove \eqref{180511e1.3} for an exponent $p$ that is as large as possible. As mentioned above, in the case $r=n$, Bourgain and Demeter \cite{BD15} and Bourgain, Demeter and Guth \cite{BDG16} proved \eqref{180511e1.3} with $p=n(n+1)$, which is the largest possible. In the case $n=2$ and $1\le r\le 2$, Demeter, Guth and Wang \cite{DGW19} studied \eqref{180511e1.3} carefully and obtained sharp estimates for every $p\ge 2$. Moreover, when $n=3$, they also obtained certain significant progress. We refer the readers to their paper for the precise statement of their result. It is worth mentioning that the case $(n,r)=(2,1)$ already appeared implicitly in Wooley \cite{Woo16} and Heath-Brown \cite{Hea15}. The authors there proved the cubic case of Vinogradov's mean value theorem (essentially the case $n=r=3$) by using the method of efficient congruencing. In this method, an estimate of the form \eqref{180511e1.3} played a crucial role (see also \cite{GLY19}).

 In other cases, certain partial results are known. When $(n,r)=(4,2)$, the estimate \eqref{180511e1.3} was obtained for $p=12$ in Bourgain \cite{Bou17}, Bourgain and Demeter \cite{BD16}. Moreover, Bourgain \cite{Bou17} applied such an estimate to the Riemann zeta function, and obtained improved bounds in the Lindel\"of hypothesis.\\

%
%
%

Next let us describe the strategy of the proof. We follow the idea of \cite{Bou17} and apply a bilinear method. First, we show that Theorem \ref{mainthm} follows from a bilinear decoupling inequality (see Proposition \ref{bilineardec}) by applying the broad-narrow analysis of Bourgain and Guth \cite{BG}. Next, we use the observation in \cite{Bou17} and \cite{BD17} to transfer a bilinear extension operator for the moment curve to a linear extension operator for certain two-dimensional surface (see \eqref{twodimmanifold}) via a change of variables (see Proposition \ref{dectwodim}). In the end, we will prove certain sharp decoupling inequalities for the resulting two-dimensional surfaces.   
%

The main obstacle in this approach is that it is difficult to obtain sharp decouplings for the above mentioned two-dimensional surfaces when the dimension $n$ is large. For $n=3$, the two-dimensional surfaces are hypersurfaces with nonzero Gaussian curvature. In this case, sharp decoupling inequalities have already been established in \cite{BD17}. For $n=4,5$, the resulting two-dimensional surfaces  have been studied in \cite{BD16}, \cite{Guo17}, and one can apply the decoupling inequalities obtained there to prove the claimed results in Theorem \ref{mainthm}.  However, for $n \geq 6$, there are no known decouplings for the related two-dimensional surfaces.

To overcome the obstacle, we apply a bootstrapping argument that essentially allows us to view the relevant two dimensional surfaces as small perturbations of certain moment surfaces. Moreover, sharp decoupling inequalities for these moment surfaces have already been proved in \cite{GZ18}. Such a bootstrapping argument can be dated back to the work of Pramanik and Seeger \cite{PS07}. To enable this bootstrapping process, we need to make sure that our surfaces satisfy certain translation-invariant properties at every scale (see Theorem \ref{linearalgebra}). This was achieved by some complicated linear algebra computations in Section \ref{section:linearalgebra}.\\
%

{\bf Notation:} Throughout the paper, the notation $E_{n,\Delta}$ will be sometimes abbreviated to $E_{\Delta}$. The number $r$ will be always $[\frac{n+1}{2}]$. We write $A \lesssim B$ if $A \leq cB$ for some constant $c>0$, $A \sim_K B$ if $c_K^{-1}A \leq B \leq c_K A$ for some $c_K>0$ depending on $K$, and $A=B+O_K(C)$ if $|A-B| \leq c_K C$ for some $c_K>0$ depending on $K$. The constants $c$ and $c_K$ will in general depend on fixed parameters such as $p,n$ and sometimes on the variable parameters $\epsilon, K$ but not the parameter $\delta$.

\subsection*{Acknowledgement}
The author would like to thank his advisor Shaoming Guo for introducing him to this problem and for his guidance throughout the project. The author also would like to thank Po-Lam Yung and Ruixiang Zhang for discussions on the linear algebra part of the paper (See Section 3).

\section{Bilinearization}
In this section, we will first reduce the linear decoupling inequality \eqref{cancellation} to the bilinear decoupling inequality \eqref{bilinaerdec2} by combining the broad-narrow analysis of Bourgain and Guth \cite{BG} and the linear decoupling inequalities of  Bourgain, Demeter and Guth \cite{BDG16} for larger balls. 
The argument of Bourgain and Guth \cite{BG} is carried out via an inductive argument on the radii of balls. 
However, in our case, a ball shrinks relatively ``fast" as we iterate because we start with a smaller ball $B^r$ (instead of $B^n$ as in \cite{BDG16}). Thus, the inductive argument does not work as efficiently as it does in \cite{BD15} or \cite{BDG16}. Instead of relying only on induction, what we will do is, after applying a ``smaller" number of steps of certain inductive hypothesis, to use the decoupling for the moment curve in \cite{BDG16} (see \eqref{bdg2} below) to decompose the frequency into the desired scale. 

Next, we will prove the bilinear decoupling inequality \eqref{bilinaerdec2}.
Instead of working with a bilinear extension operator for the moment curve, we will apply a change of variables (see \eqref{change_of_variable}) and transfer it to a linear extension operator for the two-dimensional manifold \eqref{twodimmanifold}; a very similar argument already appeared in \cite{BD16}. In the end, we will prove the desired  decoupling inequality (see \eqref{dectwodim2}) for the two dimensional manifold \eqref{twodimmanifold} in the remaining sections. 
\\

To run the method of Bourgain and Guth \cite{BG}, we need to introduce the notion of transversality.
Let $K \geq 1$ be a large number. Let $J_1,J_2$ be dyadic intervals with side length $K^{-1}$. These two intervals are called $K^{-1}$-transverse if the distance between them is greater than or equal to $K^{-1}$.

\begin{prop}\label{bilineardec}
		Let $n\geq 3$. Let $K \geq 1$ be sufficiently large. Let $J_1,J_2 \subset [0,1]$ be $K^{-1}$-transverse. For every $2 \leq p \leq p_n$, $\epsilon>0$, $0<\delta<K^{-1}$, and every integrable function $f:[0,1] \rightarrow \C$, we have
	\begin{equation}\label{bilinaerdec2}
	\Big\||E_{n, J_1}f E_{n, J_2}f|^{\frac12}\Big\|_{L^p(w_{B^r})}
	\leq C_{p,K,\epsilon} \delta^{-(\frac12-\frac1p) -\epsilon}
	\biggl(\sum_{\Delta \subset [0,1]: l(\Delta)=\delta } \|E_{n, \Delta}f\|_{L^p(w_{B^r})}^p \biggr)^{\frac1p}.
	\end{equation}
\end{prop}

The following theorem states the decoupling inequality  for the moment curve by Bourgain and Demeter \cite{BD15} and Bourgain, Demeter and Guth \cite{BDG16}. Theorem \ref{mainthm} will be deduced by combining Proposition \ref{bilineardec} and Theorem \ref{bdg}.

\begin{thm}[\cite{BD15, BDG16}]\label{bdg} Let $n \geq 2$. For every $2 \leq p \leq n(n+1)$ and $\epsilon>0$, and every integrable function $f:[0,1] \rightarrow \C$, we have
	\beq\label{bdg2}
	\|E_{n, [0, 1]} f\|_{L^p(w_{B^n})} \leq C_{n, p,\epsilon} \delta^{-(\frac12-\frac1p)-\epsilon} \Big(\sum_{\Delta\subset [0, 1]: l(\Delta)=\delta} \|E_{n, \Delta} f\|^p_{L^p(w_{B^n})}\Big)^{\frac1p}.
	\endeq
\end{thm}

\begin{proof}[Proof of Theorem \ref{mainthm} assuming Proposition 2.1.]
As $n$ is always fixed, here and below we will always abbreviate $E_{n, I}$ to $E_I$. We use the broad-narrow analysis from \cite{BG}.
For each $x \in \R^n$, we consider the collection of significant intervals, defined by
\[
\mathcal{C}(x) := \big\{R=(a+[0,K^{-1}]) \subset [0,1]: a \in K^{-1}\Z, \; |E_{[0,1]}f(x)| \leq
10^{-1}K |E_Rf(x)|  \big\}.
\]
By considering two possible cases $|\mathcal{C}(x)| \geq 3$ and $|\mathcal{C}(x)| \leq 2$, we obtain the following pointwise estimate:
\[
|E_{[0,1]}f(x)| \leq 10 \max_{l(R)=K^{-1}} |E_Rf(x)| + K \max_{\mathrm{dist}(R_1,R_2)\geq K^{-1} } \prod_{i=1}^{2}|E_{R_i}f(x)|^{\frac12}.
\]
We raise both sides of the last display to the $p$-th power, replace the \text{max} on the right hand side by an $l^p$-norm, integrate over $B^r$, and obtain
\beqq
\| E_{[0,1 ]}f\|_{L^p(w_{B^r})} \le C_p \biggl(\sum_{l(R)=K^{-1}} \|E_Rf\|_{L^p(w_{B^r})}^p \biggr)^{\frac1p	}
+C_{p,K}  \biggl(\sum_{\mathrm{dist}(R_1, R_2)\ge K^{-1}} 
\Big\|\prod_{i=1}^{2}|E_{R_i}f|^{\frac12}\Big\|_{L^p(w_{B^r})}^p \biggr)^{\frac1p}.
\endeqq
Next, we apply Proposition \ref{bilineardec} to the last term and obtain 
\beq\label{broadnarrow}
\| E_{[0,1 ]}f\|_{L^p(w_{B^r})} \le C_p \biggl(\sum_{l(R)=K^{-1}} \|E_Rf\|_{L^p(w_{B^r})}^p \biggr)^{\frac1p	}
+C_{p,K,\epsilon} \delta^{-(\frac12-\frac1p)-\epsilon} \biggl( \sum_{l(\Delta)=\delta}
\|E_{\Delta}f\|_{L^p(w_{B^r})}^p \biggr)^{\frac1p}.
\endeq

In (\ref{broadnarrow}), the last term is already of the desired form, the form of the right hand side of \eqref{cancellation}. We bound the first term on the right hand side of \eqref{broadnarrow} using an iteration argument: We rescale the interval $R$ to the whole interval $[0,1]$ and apply (\ref{broadnarrow}) again. To be more precise, let $M$ be the constant such that $K^{-M}=\delta^{r/n}$, and we will prove that 
\beq\label{BDiteration}
\begin{split}
\| E_{[0,1 ]}f\|_{L^p(w_{B^r})} & \le (C_p)^{m} \biggl(\sum_{l(R)=K^{-m}} \|E_Rf\|_{L^p(w_{B^r})}^p \biggr)^{\frac1p	}\\
& +C_{p,K,\epsilon}(C_pC_{n,p,\epsilon})^{2m} m \delta^{-(\frac12-\frac1p)-\epsilon} \biggl( \sum_{l(\Delta)=\delta}
\|E_{\Delta}f\|_{L^p(w_{B^r})}^p \biggr)^{\frac1p}
\end{split}
\endeq
for every integer $m$ with $1 \leq m \leq M$. Here, $C_{n,p,\epsilon}$ is the constant in \eqref{bdg2}. 

Note that if the interval is smaller than $\delta^{r/n}$ then by the uncertainty principle the last component of the curve $\{(t,t^2,\ldots,t^n): t\in [0,1] \}$ does not play a role on the ball $B^r$. Hence, we cannot apply an induction hypothesis anymore. Thus, we stop iterating if the side length of an interval $R$ reaches $\delta^{r/n}$.

We already proved \eqref{BDiteration} when $m=1$. Suppose that \eqref{BDiteration} holds true for some $m=m_0 <M$. 
We will show that \eqref{BDiteration} holds true for $m=m_0+1$. By the induction hypothesis, we obtain
\beq\label{inductionhypothesis}
\begin{split}
	\| E_{[0,1 ]}f\|_{L^p(w_{B^r})} & \le (C_p)^{m_0} \biggl(\sum_{l(R)=K^{-m_0}} \|E_Rf\|_{L^p(w_{B^r})}^p \biggr)^{\frac1p	}\\
	& +C_{p,K,\epsilon}(C_pC_{n,p,\epsilon})^{2m_0} m_0 \delta^{-(\frac12-\frac1p)-\epsilon} \biggl( \sum_{l(\Delta)=\delta}
	\|E_{\Delta}f\|_{L^p(w_{B^r})}^p \biggr)^{\frac1p}.
\end{split}
\endeq
We fix an interval $R$ with side length $K^{-m_0}$.
For the sake of simplicity, we assume that $R=[0,K^{-m_0}]$. We take $\gamma$ such that $\delta^{\gamma}=K^{-m_0}$. By applying a change of variables, we obtain 
\beqq
\|E_R f\|_{L^p(w_{B^r})}=\delta^{\gamma} \Big\|\int_0^1 f(\delta^{\gamma}t)
e(\delta^{\gamma}tx_1+\cdots+\delta^{n\gamma}t^nx_n)\,dt\Big\|_{L^p(w_{B^r})}.
\endeqq
By applying the change of variables
\beqq
\delta^{\gamma} x_1\mapsto x_1, \dots, \delta^{n\gamma} x_n\mapsto x_n, 
\endeqq
we obtain 
\beqq
\|E_Rf\|_{L^p(w_B^r)}=
\delta^{\gamma(1-\frac{n(n+1)}{2p})} \Big\|\int_0^1 f(\delta^{\gamma}t)e(tx_1+\cdots+t^nx_n)\,dt\Big\|_{L^p(w_{\widetilde{B}})}.
\endeqq
Here $\widetilde{B}$ is a rectangle box of dimension $\delta^{-r+\gamma}\times \dots\times \delta^{-r+n\gamma}$. 

Now we split the rectangular box $\widetilde{B}$ into balls $B'$ of radius $\delta^{-r+n\gamma}$, and apply \eqref{broadnarrow} to each $B'$. Afterwards, we raise everything to the $p$-th power, sum over $B'\subset \widetilde{B}$ and take the $p$-th root. In the end, the first term on the right hand side in \eqref{inductionhypothesis} is bounded by 
\beqq
\begin{split}
	&(C_p)^{m_0}\delta^{\gamma(1-\frac{1}{p}\frac{n(n+1)}{2})}(\sum_{l(I)=K^{-1}} \|E_{I} \widetilde{f}\|^p_{L^p(w_{\widetilde{B}})})^{\frac1p}
	\\&+C_{p,K,\epsilon}(C_p
	C_{n,p,\epsilon})^{2m_0}
	\delta^{\gamma(1-\frac{n(n+1)}{2p})}  \delta^{-(1-\frac{n\gamma}{r})(\frac{1}{2}-\frac{1}{p})-\epsilon} \Big(\sum_{l(\Delta)=\delta^{1-\frac{n\gamma}{r}}}\|E_{\Delta} \widetilde{f}\|^p_{L^p(w_{\widetilde{B}})}\Big)^{\frac1p}.
\end{split}
\endeqq
Here $\widetilde{f}(t):=f(\delta^{\gamma}t).$ We change all variables back and obtain 
\beqq
\begin{split}
	&(C_p)^{m_0+1}
	\Big(\sum_{l(I) = \delta^{\gamma}K^{-1}} \|E_I f\|^p_{L^p(w_{B^r})}\Big)^{\frac1p} 
	\\&+C_p 
	C_{p,K,\epsilon}(C_p
	C_{n,p,\epsilon})^{2m_0} \delta^{-(1-\frac{n\gamma}{r})(\frac{1}{2}-\frac{1}{p})-\epsilon} \Big(\sum_{l(\Delta)=\delta^{1-\frac{n\gamma}{r}+\gamma }}\|E_{\Delta} f\|^p_{L^p(w_{B^r})}\Big)^{\frac1p}.
\end{split}
\endeqq
To see how to further process the second term, we take $\Delta=[0, \delta^{1-\frac{n\gamma}{r}+\gamma }]$ as an example. The general case can be handled similarly after making an affine change of variables. In this case, we apply the decoupling inequality in Theorem \ref{bdg} for the moment curve $(t,t^2,\ldots,t^r)$. This can be done by viewing $x_{r+1}, \dots, x_n$ in the phase function $tx_1+\dots+t^n x_n$ of the extension operator $E_{\Delta}$ as dummy variables. As a consequence, we  obtain
\begin{displaymath}
	\begin{split}
		&\|E_{R} f\|_{L^p(w_{B^r})}   
		\leq (C_p)^{m_0+1}
		\biggl(\sum_{l(I)=\delta^{\gamma}K^{-1}} \|E_If\|_{L^p(w_{B^r})}^p \biggr)^{\frac1p}
		\\&
		+C_{p,K,\epsilon}(C_pC_{n,p,\epsilon})^{2m_0+2} \delta^{-(\frac12-\frac1p)-\epsilon} \delta^{\frac{n\gamma}{r}(\frac12-\frac1p) 
		-(\frac{n\gamma}{r}-\gamma)(1-\frac{r^2+r+2}{2p })}
		\biggl( \sum_{l(\Delta)=\delta}
		\|E_{\Delta}f\|_{L^p(w_{B^r})}^p \biggr)^{\frac1p}.
	\end{split}
\end{displaymath}
Note that for every $r \geq n/2$, we have
\[
\frac{n\gamma}{r}(\frac12-\frac1p)  -
(\frac{n\gamma}{r}-\gamma)(1-\frac{r^2+r+2}{2p }) \geq 0.
\]
Thus, we obtain
\beqq
\begin{split}
	\| E_{[0,1 ]}f\|_{L^p(w_{B^r})} & \le (C_p)^{m_0+1} \biggl(\sum_{l(R)=K^{-m_0-1}} \|E_Rf\|_{L^p(w_{B^r})}^p \biggr)^{\frac1p	}\\
	& +C_{p,K,\epsilon}(C_pC_{n,p,\epsilon})^{2m_0+2} (m_0+1) \delta^{-(\frac12-\frac1p)-\epsilon} \biggl( \sum_{l(\Delta)=\delta}
	\|E_{\Delta}f\|_{L^p(w_{B^r})}^p \biggr)^{\frac1p}.
\end{split}
\endeqq
By the induction argument, this completes the proof of \eqref{BDiteration}.\\

Recall that $K^{-M}=\delta^{r/n}$. By \eqref{BDiteration} with $m=M$, we obtain
\beqq
\begin{split}
	\| E_{[0,1 ]}f\|_{L^p(w_{B^r})} & \le
	\delta^{-\frac{r\log{C_p}}{n\log{K}} }  \biggl(\sum_{l(R)=\delta^{r/n}} \|E_Rf\|_{L^p(w_{B^r})}^p \biggr)^{\frac1p	}\\
	& +C_{p,K,\epsilon} \delta^{-\frac{2r\log{(C_pC_{n,p,\epsilon})}}{n\log{K}}} (\log_K{\delta^{-\frac{r}{n}}}) \delta^{-(\frac12-\frac1p)-\epsilon} \biggl( \sum_{l(\Delta)=\delta}
	\|E_{\Delta}f\|_{L^p(w_{B^r})}^p \biggr)^{\frac1p}.
\end{split}
\endeqq
We apply Plancherel's theorem and a trivial bound at $L^{\infty}$ to control the first term on the right hand side. In the end, we will take $K$ to be large enough and obtain
\[
\| E_{[0,1 ]}f\|_{L^p(w_{B^r})} \leq
\tilde{C}_{p,K,\epsilon}\delta^{-\epsilon} \bigl(  \delta^{-(1-\frac{r}{n})(1-\frac{2}{p}) }+\delta^{-(\frac12-\frac1p)} \bigr) \biggl( \sum_{l(\Delta)=\delta}
\|E_{\Delta}f\|_{L^p(w_{B^r})}^p \biggr)^{\frac1p}.
\]
It suffices to note that $$(1-\frac{r}{n})(1-\frac{2}{p}) \leq \frac12-\frac1p$$ for every $r \geq n/2$ and $p \geq 2$. This finishes the proof of Theorem \ref{mainthm} assuming Proposition 2.1.
\end{proof}
%
We prove Proposition \ref{bilineardec} in the next step. In previous decoupling papers \cite{BD15, BDG16}, a large separation of intervals (the transversality constant $K^{-1}$) is essential  because of the use of multilinear Kakeya inequalities. However, in this paper, we do not (directly) use any multilinear Kakeya inequality. In fact, we will see that there is certain significant advantage if the separation of intervals is small (see the statement of Proposition \ref{scalebroadnarrow}). 

This phenomenon is particular to the approach we are using: We will apply a change of variables (see \eqref{change_of_variable}) to transfer the problem of bilinear decoupling for the moment curve to the problem of linear decoupling for  a two dimensional manifold (given by \eqref{twodimmanifold}). This change of variables is non-linear. As a consequence, it is hard to find an explicit expression of the manifold, not  to mention to prove certain sharp decoupling inequalities. However, we will see that the smaller the transversality constant is, the more the induced manifold will behave like a monimial manifold. Moreover, a sharp decoupling inequality for such a moment manifold has already been established in \cite{GZ18}.

The following proposition states a bilinear decoupling inequality with a smaller transversality constant, compared with the one in Proposition \ref{bilineardec}.

\begin{prop}\label{scalebroadnarrow}
	Let $n \geq 3$ and $\epsilon>0$. Let $K \geq 1$ be sufficiently large. Let $0< \delta<K^{-1}$. Let $J_1,J_2 \subset [0,\delta^{\epsilon}]$ be $\delta^{\epsilon}K^{-1}$-transverse. 
	For every $2 \leq p \leq p_n$ and every integrable function $f:[0,1] \rightarrow \C$, we have
	\[
	\Big\||E_{J_1}f E_{J_2}f|^{\frac12}\Big\|_{L^p(w_{B^r})}
	\leq C_{p,K,\epsilon} \delta^{-(\frac12-\frac1p) -C\epsilon}
	\biggl(\sum_{\Delta \subset [0,\delta^{\epsilon}]: l(\Delta)=\delta } \|E_{\Delta}f\|_{L^p(w_{B^r})}^p \biggr)^{\frac1p}.
	\]
\end{prop}
Proposition \ref{bilineardec} follows from Proposition \ref{scalebroadnarrow} via a simple scaling argument. We leave out the details. 

It remains to prove Proposition \ref{scalebroadnarrow}. Given two intervals $J_1, J_2\subset [0, \delta^{\epsilon}]$ that are $\delta^{\epsilon} K^{-1}$-transverse, we follow the idea of convolving two measures that are supported on $J_1$ and $J_2$ separately, and consider the support of the output measure
\beq\label{190924e2.7}
\{(t+s, t^2+s^2, \dots, t^n+s^n): t\in J_1, s\in J_2\}.
\endeq
Define 
\beq\label{change_of_variable}
u(t, s):=t+s,\ \  v(t, s):=t^2+s^2,
\endeq
and the set 
\[
L(J_1,J_2):=\big\{ (t+s,t^2+s^2): (t, s) \in J_1 \times J_2  \big\}.
\]
Under the assumption that $J_1$ and $J_2$ are $\delta^{\epsilon} K^{-1}$-transverse, it is not difficult to see that the Jacobian matrix $\frac{\partial(u, v)}{\partial(t, s)}$ is non-singular on $J_1\times J_2$. 
 This allows us to write $t$ and $s$ as functions of $u$ and $v$. Furthermore, we can write \eqref{190924e2.7} as 
\beq\label{twodimmanifold}
\mathcal{M}=
\big\{(u,v,p_3(u,v),\ldots, p_n(u,v): (u,v) \in L(J_1,J_2))
\big\}
\endeq
where 
\beq\label{newton}
\begin{split}
p_k(u,v)&:=t(u, v)^k+s(u, v)^k \;\;\; \mathrm{for} \;\;\; k=1,2,\ldots,n.
\end{split}
\endeq

Given an integer $n'\ge 3$, smooth functions $P_3, \dots, P_{n'}$, and a surface 
\beqq
\mc{M}':=\{(u, v, P_3(u,v), \ldots,  P_{n'}(u,v)\},
\endeqq
we define the associated extension operator 
\beqq
E^{\mc{M}'}_{\Box}g(x):=\int_{\Box} g(u, v) e(ux_1+ vx_2+P_3(u, v) x_3+\dots+P_{n'}(u, v) x_{n'})dudv
\endeqq
for $x'\in \R^{n'}$ and a set $\Box\subset \R^2$. 
Proposition \ref{scalebroadnarrow} follows from the decoupling for the two-dimensional surface $\mathcal{M}$ given by \eqref{twodimmanifold}.
\begin{prop}\label{dectwodim}
	Let $n \geq 3$ and $\epsilon>0$. Let $K \geq 1$ be sufficiently large. Let $0<\delta<K^{-1}$. Let $J_1,J_2 \subset [0,\delta^{\epsilon}]$ be $\delta^{\epsilon}K^{-1}$-transverse.
	For every $2 \leq p \leq p_n/2$ and every integrable function $g:[0,1]^2 \rightarrow \C$, we have
	\beq\label{dectwodim2}
	\big\|E_{L(J_1,J_2)}^{\mathcal{M}}g\big\|_{L^{p}(w_{B^r})}
	\leq C_{p,K,\epsilon} \delta^{-2(\frac12-\frac{1}{p})-C\epsilon }
	\biggl(\sum_{\square : \square \cap L(J_1,J_2) \neq \emptyset,\, l(\square)=\delta } \big\|E_{\square}^{\mathcal{M}}g\big\|_{L^{p}(w_{B^r})}^{p} \biggr)^{\frac{1}{p}}.
	\endeq
	Here, the sum runs over squares $\square$ of the form $[a,a+\delta] \times [b,b+\delta]$ with $a,b \in \delta\Z$ that have non-empty intersection with $L(J_1,J_2)$.
\end{prop}

\begin{proof}[Proof of Proposition \ref{scalebroadnarrow} assuming Proposition \ref{dectwodim}.]

Consider a collection $\{L(\Delta_1,\Delta_2) \}_{\substack{\Delta_1 \subset J_1, \Delta_2 \subset J_2}}$, where $\Delta_i \subset [0,1]$ is a dyadic interval with side length $\delta$.

It is not difficult to find a collection of $10^4$ square grids $\{\mathcal{G}_i\}_{1 \leq i \leq 10^4 }$ satisfying the followings:
\begin{enumerate}
	\item Each square in each grid $\mathcal{G}_i$ has a dyadic side length  $16\delta$,
	\item For every $\Delta_1,\Delta_2$, there exists $i$ such that ${L}({\Delta_1,\Delta_2})$ is contained in a square from $\mathcal{G}_i$.
\end{enumerate}
Also, a simple computation shows that there exists a small positive constant $c_K$ independent of the choice of $\Delta_1,\Delta_2$ and the parameter $\delta$ such that
\beq\label{smallballcontaining}
B_{\Delta_1,\Delta_2}:=
B\big((X,Y),c_K\delta^{1+\epsilon}\big)
\subset L(\Delta_1,\Delta_2),
\endeq
where
\[
\Delta_1=[a,a+\delta], \;\; \Delta_2=[b,b+\delta]
\]
for some $a,b$
and $X = a+b+\delta, Y = a^2+b^2+(b-a)\delta+\delta^2.$
Here, $B\big((X,Y),r \big)$ denotes the ball of radius $r$ centered at the point $(X,Y)$. 
We denote by $Q_{\Delta_1,\Delta_2}$ the square from some grid $\mathcal{G}_i$ that contains $L(\Delta_1,\Delta_2)$. By the property \eqref{smallballcontaining}, for each square from a grid $\mathcal{G}_i$, the number of sets of the form $L(\Delta_1,\Delta_2)$ that intersect such a square is $O_K(\delta^{\epsilon})$.

We use the change of variables: $u=t+s$ and $v=t^2+s^2$. Let $g(u,v)=f(t)f(s)(|\det{\frac{\partial (u,v)}{\partial (t,s)}}|)^{-1}$. Then
\beqq
\begin{split}
	E_{J_1}f(x)E_{J_2}f(x)
	&=\int_{J_1}\int_{J_2}f(t)f(s)e((t+s)x_1+(t^2+s^2)x_2+\cdots+(t^n+s^n)x_n)\,dtds
	\\&=\int_{L(J_1,J_2)}g(u,v)e(ux_1+vx_2+p_3(u,v)x_3+\cdots+p_n(u,v)x_n )\,dudv
	\\&=E^{\mathcal{M}}_{L(J_1,J_2)}g(x).
\end{split}
\endeqq
Since
\[
\Big\||E_{J_1}f E_{J_2}f|^{\frac12}\Big\|_{L^{p_n}(w_{B^r})}  = \big\|E_{L(J_1,J_2)}^{\mathcal{M}}g\big\|_{L^{p_n/2}(w_{B^r})}^{1/2},
\]
Proposition \ref{scalebroadnarrow} follows from
\beq\label{reduction}
	\big\|E_{L(J_1,J_2)}^{\mathcal{M}}g\big\|_{L^{p_n/2}(w_{B^r})}^{1/2}
\leq C_{K,\epsilon} \delta^{-(\frac12-\frac{1}{p_n}) -C\epsilon}
\biggl(\sum_{\Delta \subset [0,\delta^{\epsilon}]: l(\Delta)=\delta } \|E_{\Delta}f\|_{L^{p_n}(w_{B^r})}^{p_n} \biggr)^{\frac{1}{p_n}}.
\endeq
By using the grids constructed at the beginning of the proof, we obtain
\beqq\label{grideq}
E_{L(J_1,J_2)}^{\mathcal{M}}g = \sum_{i=1}^{10^4}
\sum_{\substack{\Delta_1,\Delta_2: Q_{\Delta_1,\Delta_2} \in \mathcal{G}_i\\  l(\Delta_1)=l(\Delta_2)=\delta}  } E^{\mathcal{M}}_{L(\Delta_1,\Delta_2)}g. 
\endeqq
Therefore, 
\beq\label{2.17}
\big\|E_{L(J_1,J_2)}^{\mathcal{M}}g\big\|_{L^{p_n/2}(w_{B^r})}^{1/2}
\lesssim \sup_{1 \leq i \leq 10^4} \biggl( \bigg\|\sum_{\substack{\Delta_1,\Delta_2: Q_{\Delta_1,\Delta_2} \in \mathcal{G}_i\\  l(\Delta_1)=l(\Delta_2)=\delta}  } E^{\mathcal{M}}_{L(\Delta_1,\Delta_2)}g \bigg\|_{L^{p_n/2}(w_{B^r})}^{1/2} \biggr).
\endeq
We apply Proposition \ref{dectwodim} and bound \eqref{2.17}  by
\begin{gather*}
C_{K,\epsilon} \delta^{-(\frac12-\frac{2}{p_n})-C\epsilon }
\sup_{1 \leq i \leq 10^4}
\biggl(\sum_{\square \in \mathcal{G}_i : \square \cap L(J_1,J_2) \neq \phi,\, l(\square)=16\delta } \Big\|
\sum_{\substack{\Delta_1,\Delta_2: Q_{\Delta_1,\Delta_2}=\Box\\  l(\Delta_1)=l(\Delta_2)=\delta}  }
E_{L(\Delta_1,\Delta_2)}^{\mathcal{M}}g\Big\|_{L^{p_n/2}(w_{B^r})}^{p_n/2} \biggr)^{\frac{1}{p_n}}.
\end{gather*}
By the property that $|\{ (\Delta_1,\Delta_2): Q_{\Delta_1,\Delta_2}=\square  \}| = O(\delta^{-\epsilon})$ for each $\square \in \mathcal{G}_i$, the last expression can be further bounded by
\[
CC_{K,\epsilon} \delta^{-(\frac12-\frac{2}{p_n})-C'\epsilon }
\biggl(\sum_{l(\Delta_1)=l(\Delta_2)=\delta } \Big\|
E_{L(\Delta_1,\Delta_2)}^{\mathcal{M}}g\Big\|_{L^{p_n/2}(w_{B^r})}^{p_n/2} \biggr)^{1/p_n}.
\]
Since
\[
\Big\|
E_{L(\Delta_1,\Delta_2)}^{\mathcal{M}}g\Big\|_{L^{p_n/2}(w_{B^r})}^{\frac12} = \Big\|\big|E_{\Delta_1}f E_{\Delta_2}f\big|^{\frac12}\Big\|_{L^{p_n}(w_{B^r})},
\]
we obtain
\[
\big\|E_{L(J_1,J_2)}^{\mathcal{M}}g\big\|_{L^{p_n/2}(w_{B^r})}^{1/2} 
\leq
CC_{K,\epsilon} \delta^{-(\frac12-\frac{2}{p_n})-C'\epsilon }
\biggl(\sum_{ l(\Delta_1)=l(\Delta_2)=\delta } \Big\|\big|E_{\Delta_1}f E_{\Delta_2}f\big|^{\frac12}\Big\|_{L^{p_n}(w_{B^r})}^{p_n} \biggr)^{1/p_n}.
\]
By applying the Cauchy-Schwarz inequality twice, the right hand side can be bounded by
\[
\begin{split}
 & CC_{K,\epsilon}
\delta^{-(\frac12-\frac{2}{p_n})-C'\epsilon }
\biggl(\sum_{l(\Delta_1)=l(\Delta_2)=\delta } \|E_{\Delta_1}f  \|_{L^{p_n}(w_{B^r})}^{p_n/2}
\|
E_{\Delta_2}f
\|_{L^{p_n}(w_{B^r})}^{p_n/2}
 \biggr)^{1/p_n}
 \\&
 \leq CC_{K,\epsilon}
 \delta^{-(\frac12-\frac{2}{p_n})-C'\epsilon }
 \biggl(\sum_{\Delta \subset [0,\delta^\epsilon]: l(\Delta)=\delta } \|E_{\Delta}f  \|_{L^{p_n}(w_{B^r})}^{p_n/2}
 \biggr)^{2/p_n}
\\& \leq
C_0C_{K,\epsilon}
\delta^{-(\frac12-\frac{1}{p_n})-C'\epsilon }
\biggl(\sum_{\Delta \subset [0,\delta^{\epsilon}] : l(\Delta)=\delta } \|
E_{\Delta}f \|_{L^{p_n}(w_{B^r})}^{p_n} \biggr)^{1/p_n}.
\end{split}
\]
Therefore, we obtain the inequality $\eqref{reduction}$ and this completes the proof of Proposition \ref{scalebroadnarrow} assuming Proposition 2.4.
\end{proof}

For the rest of the paper, we give a proof of Proposition \ref{dectwodim}.

\section{Some linear algebra}\label{section:linearalgebra}
In this section, we will make some preparation for the proof of Proposition \ref{dectwodim}. To be more precise, we will show that, after certain affine transformations, the manifold $\mc{M}$ (defined in \eqref{twodimmanifold}) is very close to some moment manifold (see Theorem \ref{linearalgebra} and \eqref{monomial_manifold}). \\
%

For each $(a,b) \in [0,1]^2$, we define the manifold $\M_{(a,b)}$ to be
\[
\M_{(a,b)} = \big\{ (u,v,p_3(u+a,v+b),\ldots,p_n(u+a,v+b))  \big\}.
\]
Here, the polynomials $p_i$ are defined in \eqref{newton}.
Next we define a relation between two manifolds. For $i=1,2$, let $\M_i$ be a manifold given by
\[
\M_i = \{(u,v,P_{3,i}(u,v),\ldots,P_{n,i}(u,v)) \},
\]
where $P_{j,i}$ is a real-valued function for each $i$ and $j$. We say that $\M_1 \cong_{M} \M_2$ if there exist an invertible linear transformation 
$M:\R^n \rightarrow \R^n$ and some vector ${\bf{b}} \in \R^n$ such that
\[
(u,v,P_{3,2}(u,v),\ldots,P_{n,2}(u,v))^{\intercal} = M\big(u,v,P_{3,1}(u,v),\ldots,P_{n,1}(u,v)\big)^{\intercal}+\bf{b}^{\intercal}
\]
for all $u,v \in \R$. Here, the superscript $\intercal$ refers to a transpose.\\

The main result in this section is the following. Recall that the functions $u,v$ are defined in \eqref{change_of_variable}.

\begin{thm}\label{linearalgebra}
	Let $n \geq 3$, $K \geq 100$ and $0 \leq \zeta \leq 1$.
	Suppose that  $(a,b)=(u(\alpha,\beta),v(\alpha,\beta))$ for some $0 \leq \alpha,\beta \leq \zeta$ with $|\alpha -\beta| \geq \zeta K^{-1}$. Then there exists an invertible linear transformation $M$ such that 
	\[
	\mathcal{M}_{(a,b)} \cong_M \big\{ (u,v,q_3(u,v),\ldots, q_n(u,v )) \big\},
	\]
	where
	\begin{displaymath}\label{desired}
	q_i(u,v) = \left\{ \begin{array}{ll}
	\frac{2k+1}{2^{k}}uv^{k} + O_{K}\big(\zeta|(u,v)|^{k+1}+ |(u,v)|^{k+2} \big)& \textrm{with $k=\frac{i-1}{2}$ when $i$ is odd}\\[3pt]
	\frac{1}{2^{k}}v^{k+1}+ O_{K}\big(\zeta|(u,v)|^{k+1} + |(u,v)|^{k+2} \big)& \textrm{with $k=\frac{i-2}{2}$ when $i$ is even}.
	\end{array} \right.
	\end{displaymath}
	Moreover, $\det{M} \sim_K 1$.
\end{thm}
\begin{proof}[Proof of Theorem \ref{linearalgebra}.]
	
	For each $i \in \N$, the gradient $\nabla_{u,v}^{i} $ is defined by
\beq\label{grad}
\nabla_{u,v}^i (f) :=  \begin{pmatrix}
\frac{\partial^i f}{\partial u^i} & \frac{\partial^i f}{\partial u^{i-1} \partial v } & \cdots  & \frac{\partial^i f}{\partial v^i}
\end{pmatrix}^\intercal.
\endeq
To obtain good approximation formula for $p_k(u+a,v+b)$, a natural idea is to apply Taylor's expansion. Taking partial derivatives of $p_k$ in terms of $u$ and $v$ can get very complicated. We will instead compute partial derivatives of $p_k$ in terms of $t$ and $s$, and then apply formulas for derivatives of implicit functions. In this approach, the following lemma will be particularly useful.   

%
\begin{lem}\label{changeofvariables}
	Let $k \geq 3$.
	Let $u(t,s)=t+s$ and $v(t,s)=t^2+s^2$.
	Suppose that $(a,b)=(u(\alpha,\beta),v(\alpha,\beta))$ for some $\alpha,\beta$ with $\alpha \neq \beta$.
	Then there exists an invertible matrix $A_{a,b}$ depending on $a,b$ such that 
	\[
	\begin{pmatrix}
	\nabla_{u,v} & \nabla^2_{u,v} & \cdots & \nabla^k_{u,v}
	\end{pmatrix}^\intercal f |_{(a,b)}
	= \Big(A_{a,b}\cdot \begin{pmatrix}
	\nabla_{t,s} & \nabla^2_{t,s} & \cdots & \nabla^k_{t,s}
	\end{pmatrix}^\intercal\Big) g \Big|_{(\alpha,\beta)},
	\]
	for every smooth function $f(u, v)$ and $g(t, s):=f(u(t, s), v(t, s))$.	
\end{lem}

\begin{proof}[Proof of Lemma \ref{changeofvariables} ]
We first show that there exists a matrix $B_{\alpha,\beta}$ such that 
\[
\begin{pmatrix}
\nabla_{t,s} & \nabla^2_{t,s} & \cdots & \nabla^k_{t,s}
\end{pmatrix}^\intercal g|_{(\alpha,\beta)}
= \Big(B_{\alpha,\beta} \cdot
\begin{pmatrix}
\nabla_{u,v} & \nabla^2_{u,v} & \cdots & \nabla^k_{u,v}
\end{pmatrix}^\intercal \Big)f|_{(a,b)}.
\]
Afterwards we will show that the matrix $B_{\alpha,\beta}$ is invertible. In the end, we can take $A_{a, b}$ to be the inverse of $B_{\alpha, \beta}$.\\

By the chain rule, we obtain
\[
\begin{split}
&\frac{\partial}{\partial t} = \frac{\partial u}{\partial t}\frac{\partial }{\partial u} + \frac{\partial v}{\partial t}\frac{\partial }{\partial v} = \frac{\partial}{\partial u}+2t \frac{\partial}{\partial v}, \\&
\frac{\partial}{\partial s} = \frac{\partial u}{\partial s}\frac{\partial }{\partial u} + \frac{\partial v}{\partial s}\frac{\partial }{\partial v} = \frac{\partial}{\partial u}+2s \frac{\partial}{\partial v}.
\end{split}
\]
By direct computations, we obtain
\[
\large\begin{pmatrix}
\frac{\partial}{\partial t} \\[0.50em] \frac{\partial}{\partial s} \\[0.50em] \frac{\partial^2}{\partial t^2}\\[0.50em]
\frac{\partial^2}{\partial t \partial s}\\[0.50em]
\frac{\partial^2}{\partial s^2}\\[0.50em]
\vdots \\[0.50em]
\frac{\partial^k}{\partial t^k} \\[0.50em]
\vdots \\[0.50em]
\frac{\partial^k}{\partial s^k} \\[0.50em]
\end{pmatrix} =
\large\begin{pmatrix}
1 & 2\alpha & 0 & 0 & 0 &\cdots &0 & \cdots &0 \\[0.50em]
1 & 2\beta & 0 & 0 & 0  &\cdots &0& \cdots &0 \\[0.50em]
* & * & 1 & 4\alpha & 4\alpha^2 & \cdots &0 & \cdots & 0\\[0.50em]
* & * & 1 & 2\alpha+2\beta & 4\alpha \beta & \cdots &0&\cdots &0 \\[0.50em]
* & * & 1 & 4\beta & 4\beta^2 & \cdots &0&\cdots &0 \\[0.50em]
\vdots & \vdots & \vdots & \vdots & \vdots & \ddots &\vdots &\vdots &\vdots \\[0.50em]
* & * & * & * & * & \cdots & 1 &\cdots & (2\alpha)^k \\[0.50em]	
\vdots & \vdots & \vdots & \vdots & \vdots & \vdots & \vdots &\vdots & \vdots \\[0.50em]
* & * & * & * & * & \cdots & 1 &\cdots & (2\beta)^k \\[0.50em]
\end{pmatrix}
\large\begin{pmatrix}
\frac{\partial}{\partial u} \\[0.50em] \frac{\partial}{\partial v} \\[0.50em] \frac{\partial^2}{\partial u^2}\\[0.50em]
\frac{\partial^2}{\partial u \partial v}\\[0.50em]
\frac{\partial^2}{\partial v^2}\\[0.50em]
\vdots \\[0.50em]
\frac{\partial^k}{\partial u^k} \\[0.50em]
\vdots \\[0.50em]
\frac{\partial^k}{\partial v^k} \\[0.50em]
\end{pmatrix}.
\]
Here, every $*$ denotes a number that we will not keep track of.
The matrix $B_{\alpha,\beta}$ has the following form
\[
B_{\alpha,\beta} = \begin{pmatrix}
A_{2 \times 2} & 0_{2 \times 3} & \cdots &0_{2 \times k} \\
* & A_{3 \times 3} & \cdots & 0_{3 \times k} \\
\vdots & \vdots & \ddots & \vdots \\
* & * & \cdots & A_{k,k}
\end{pmatrix}.
\]
Here, $A_{i \times i}$ is an $i \times i$ matrix and $0_{i \times j}$ is an $i \times j$ matrix whose components are all zero. Thus, to prove Lemma \ref{changeofvariables}, it suffices to show that $\det{(A_{i \times i})} \neq 0$ for all $i$.\\

We define the polynomials $$r_{j}(t,s) := (t+2\alpha s)^{i-1-j}(t+2\beta s)^{j}$$ for $j=0,\ldots,i-1$. Then the matrix $(A_{i\times i})^{\intercal}$ can be expressed as
\[
(A_{i \times i})^{\intercal} = 
\large\begin{pmatrix}
\frac{1}{(i-1)!0!}
\frac{\partial^{i-1} }{\partial t^{i-1} } \\[0.50em] 
\frac{1}{(i-2)!1!}
\frac{\partial^{i-1} }{\partial t^{i-2} \partial s } \\[0.50em]  \frac{1}{(i-3)!2!}
\frac{\partial^{i-1} }{\partial t^{i-3} \partial s^2 }\\[0.5em] \vdots \\[0.50em] 
\frac{1}{0!(i-1)!}
\frac{\partial^{i-1} }{\partial s^{i-1} }
\end{pmatrix}
\large\begin{pmatrix}
r_0 & r_1 & r_2 & \cdots & r_{i-1}
\end{pmatrix}.
\]
Without loss of generality, we may assume that $\beta \neq 0$.
We apply a change of variables:
\[
\begin{split}
t &\mapsto t =:\bar t  \\
s &\mapsto t+2\beta s =:\bar s.
\end{split}
\]
By the chain rule, we obtain
\[
\begin{split}
&\frac{\partial}{\partial t}=\frac{\partial \bar{t}}{\partial t}\frac{\partial}{\partial \bar{t}}+\frac{\partial \bar{s}}{\partial t}\frac{\partial}{\partial \bar{s}}=\frac{\partial}{\partial \bar{t}}+\frac{\partial}{\partial \bar{s}},
\\&
\frac{\partial }{\partial s}=
\frac{\partial \bar{t}}{\partial s}\frac{\partial}{\partial \bar{t}}+\frac{\partial \bar{s}}{\partial s}\frac{\partial}{\partial \bar{s}}=2\beta \frac{\partial}{\partial \bar{s}}.
\end{split}
\]
By direct computations, we obtain
\[
(A_{i\times i})^{\intercal} = 
\large\begin{pmatrix}
1 & * & * & \cdots & * \\
0 & 2\beta & * & \cdots & * \\
0 & 0 & 2^2\beta^2 & \cdots & * \\
\vdots & \vdots & \vdots & \ddots & \vdots \\
0 & 0 & 0 & \cdots & 2^{i-1}\beta^{i-1}
\end{pmatrix}
\large\begin{pmatrix}
\frac{1}{(i-1)!0!}
\frac{\partial^{i-1} }{\partial \bar t^{i-1} } \\[0.50em] 
\frac{1}{(i-2)!1!}
\frac{\partial^{i-1} }{\partial \bar t^{i-2} \partial \bar s } \\[0.50em]  \frac{1}{(i-3)!2!}
\frac{\partial^{i-1} }{\partial \bar t^{i-3} \partial \bar s^2 }\\[0.5em] \vdots \\[0.50em] 
\frac{1}{0!(i-1)!}
\frac{\partial^{i-1} }{\partial \bar s^{i-1} }
\end{pmatrix}
\large\begin{pmatrix}
r_0 & r_1 & r_2 & \cdots & r_{i-1}
\end{pmatrix}.
\]
We compute the derivatives of $r_j$ and obtain
\[
(A_{i\times i})^{\intercal}
=\large\begin{pmatrix}
1 & * & * & \cdots & * \\
0 & 2\beta & * & \cdots & * \\
\vdots & \vdots & \vdots & \ddots & \vdots \\
0 & 0 & 0 & \cdots & (2\beta)^{i-1}
\end{pmatrix}
\large\begin{pmatrix}
(1-\frac{2\alpha}{2\beta})^{i-1} & 0 & 0 & \cdots & 0 \\
* & (1-\frac{2\alpha}{2\beta})^{i-2} & 0 & \cdots & 0 \\
\vdots & \vdots & \vdots & \ddots & \vdots \\
* & * & * & \cdots & (1-\frac{2\alpha}{2\beta})^{0}
\end{pmatrix}.
\]
Thus, we obtain $\det{(A_{i\times i})} = (2\beta-2\alpha)^{\frac{i(i-1)}{2}}$, and this is non-vanishing whenever $\alpha \neq \beta$. This completes the proof of the lemma.
\end{proof}

Let us continue with the proof of Theorem \ref{linearalgebra}. We first consider the case $\zeta=0$. Note that in this case $(a,b)=(0,0)$, and what we need to prove becomes 
\begin{equation}\label{leadingterm}
	p_i(u,v) =\left\{ \begin{array}{ll}
		\frac{2k+1}{2^{k}}uv^{k} + e_{2k+1}(u,v)& \textrm{with $k=\frac{i-1}{2}$ when $i$ is odd}\\[3pt]
		\frac{1}{2^{k}}v^{k+1}+ \bar e_{2k+1}(u,v)& \textrm{with $k=\frac{i-2}{2}$ when $i$ is even}.
	\end{array} \right.
\end{equation}
Here, for $k \geq 1$, $e_{2k+1}, \bar e_{2k+1}$ are some polynomials whose lowest degree is greater than or equal to $k+2$. The functions $e_1,\bar e_1$ are defined to be identically zero. Here, $\bar e_{2k+1}$ does not indicate the complex conjugation of $e_{2k+1}$. 

We prove \eqref{leadingterm} by an inductive argument. The base cases $i=1,2$ of the induction are trivial. Note that in this case $k=0$. Next, by Newton's identity, for every $i \geq 3$, we have 
\[
p_{i}(u,v)=up_{i-1}(u,v)-(\frac{u^2-v}{2})p_{i-2}(u,v).
\]
Suppose that $k_0 \geq 0$ and \eqref{leadingterm} holds true for all $k$ with $ 0 \leq k \leq 2k_0$. We apply the above identity and the induction hypothesis, and obtain 
\[
\begin{split}
p_{2k_0+1}(u,v) &= up_{2k_0}(u,v) - \big(\frac{u^2-v}{2}\big)p_{2k_0-1}(u,v)
\\& = \frac{1}{2^{k_0-1}}uv^{k_0}+u\bar{e}_{2k_0-1}(u,v)
-\big(\frac{u^2-v}{2}\big)\bigg(\frac{2k_0-1}{2^{k_0-1}}uv^{k_0-1}+e_{2k_0-1}(u,v) \bigg)
\\& = \frac{2k_0+1}{2^{k_0}} uv^{k_0} + \biggl( -\frac{2k_0-1}{2^{k_0}}u^3v^{k_0-1} +u\bar e_{2k_0-1}(u,v)-(\frac{u^2-v}{2})e_{2k_0-1}(u,v) \biggr)
\\& =: \frac{2k_0+1}{2^{k_0}} uv^{k_0} + e_{2k_0+1}(u,v),
\end{split}
\]
and
\[
\begin{split}
p_{2k_0+2}(u,v) &= up_{2k_0+1}(u,v) - \big(\frac{u^2-v}{2}\big)p_{2k_0}(u,v)
\\&= u \bigg( \frac{2k_0+1}{2^{k_0}}uv^{k_0}+
e_{2k_0+1}(u,v) \bigg)-\big( \frac{u^2-v}{2}\big)
\bigg( \frac{1}{2^{k_0-1}}v^{k_0} +\bar{e}_{2k_0-1}(u,v) \bigg)
\\& = \frac{1}{2^{k_0}}v^{k_0+1} + \biggl( u \Big(
\frac{2k_0+1}{2^{k_0}}uv^{k_0} + e_{2k_0+1}(u,v) \Big)
-\frac{u^2v^{k_0}}{2^{k_0}}
-(\frac{u^2-v}{2})\bar e_{2k_0-1}(u,v) \biggr)
\\& =: \frac{1}{2^{k_0}}v^{k_0+1} + \bar e_{2k_0+1}(u,v).
\end{split}
\]
Note that $e_{2k_0+1},\bar e_{2k_0+1}$ are  polynomials whose lowest degrees are at least $k_0+2$. This closes the induction, and therefore finishes the proof of \eqref{leadingterm}.\\

Next, we consider the case that $ (a,b)=(u(\alpha,\beta),v(\alpha,\beta))$ for some $0 \leq \alpha,\beta \leq \zeta$ with $|\alpha -\beta| \geq \zeta K^{-1}$, where $\zeta>0$. 
Let $h$ be an arbitrary polynomial of two variables $u,v$:
$$
h(u,v) = \sum_{j=0}^{\infty}\sum_{i=0}^{j}c_{i,j,k}u^iv^{j-i}.
$$
We define a truncation of the polynomial $h$ at the degree $l$ by
$$
(h)_{l}(u,v) := \sum_{j=0}^{l}\sum_{i=0}^{j}c_{i,j,k}u^iv^{j-i}.
$$
For every function $g:\R^2 \rightarrow \C$ and $a,b \in \R$, we define the function $g^{a,b}(u,v)$ to be 
\begin{displaymath}
g^{a,b}(u,v):=g(a+u,b+v).
\end{displaymath}
We will show that for every $k \geq 1$ and $j=0,1$ there exist $w_{j,k}=(w_{1,j,k},\cdots,w_{2k,j,k}) \in \R^{2k}$ with $|w_{i,j,k}| \lesssim_{K} \zeta$ and some constant $C_{a,b,2k+1+j}$ such that
\beq\label{existence}
(p_{2k+1+j}^{a,b})_k(u,v) = C_{a,b,2k+1+j}+ \sum_{i=1}^{2k}
w_{i,j,k}(p_i^{a,b})_k(u,v)
\endeq
holds for every $u, v$. Let us first accept this claim. 
By an affine transformation, we can replace the surface $\M_{(a,b)}$ by
\[
\{(u,v,p_3'(u,v),\ldots,p_n'(u,v) ) \},
\]
where
\[
p_{2k+1+j}'(u,v) := p_{2k+1+j}^{a,b}(u,v)-\sum_{i=1}^{2k}w_{i,j,k}(p_i^{a,b})(u,v)
\]
for $j=0,1$ and $k \geq 1$. Here, $w_{j,k}=(w_{1,j,k},\ldots,w_{2k,j,k})$ is the vector satisfying the claim \eqref{existence}.
By the claim \eqref{existence}, we obtain
\begin{displaymath}
\begin{split}
p'_{2k+1+j}(u,v)&=
p_{2k+1+j}^{a,b}(u,v) - \sum_{i=1}^{2k} w_{i,j,k} (p_i^{a,b})(u,v)
\\&
= C_{a,b,2k+1+j}+
\big( (p_{2k+1+j}^{a,b})(u,v)- (p_{2k+1+j}^{a,b})_{k}(u,v) \big) - \sum_{i=1}^{2k} w_{i,j,k} \big( (p_i^{a,b})(u,v)-(p_i^{a,b})_{k}(u,v) \big)
\\&
= C_{a,b,2k+1+j}+ (p_{2k+1+j}^{a,b})(u,v)- (p_{2k+1+j}^{a,b})_{k}(u,v)
+O_K(\zeta |(u,v)|^{k+1}).
\end{split}
\end{displaymath}
Note that the error is harmless. 
We first consider the case when $j=1$. By \eqref{leadingterm}, we obtain
\[
(p_{2k+2}^{a,b})(u,v)-(p_{2k+2}^{a,b})_k(u,v)
=\frac{1}{2^k}v^{k+1}+(\bar{e}_{2k+1}^{a,b})(u,v)-(\bar{e}_{2k+1}^{a,b})_k(u,v).
\]
Since $\bar{e}_{2k+1}^{a,b}$ is a polynomial of degree greater than or equal to $k+2$ and $|a|,|b| \lesssim \zeta$, we obtain
\[
(\bar{e}_{2k+1}^{a,b})(u,v)-(\bar{e}_{2k+1}^{a,b})_k(u,v)=
O_K(\zeta |(u,v)|^{k+1} + |(u,v)|^{k+2} ).
\]
Hence, we finally obtain
\[
p'_{2k+2}(u,v)=C_{a,b,2k+2}+\frac{1}{2^k}v^{k+1}+O_K(\zeta |(u,v)|^{k+1} + |(u,v)|^{k+2} ).
\]
We next consider the case when $j=0$. By \eqref{leadingterm}, we obtain
\[
(p_{2k+1}^{a,b})(u,v)-(p_{2k+1}^{a,b})_{k}(u,v)=
\frac{2k+1}{2^{k}}uv^{k}+({e}_{2k+1}^{a,b})(u,v)-
({e}_{2k+1}^{a,b})_{k}(u,v).
\]
Since ${e}_{2k+1}^{a,b}$ is a polynomial of degree greater or equal to $k+2$, we get
\[
({e}_{2k+1}^{a,b})(u,v)-
({e}_{2k+1}^{a,b})_{k}(u,v)
=O_K(\zeta |(u,v)|^{k+1} + |(u,v)|^{k+2} ).
\]
Hence, we finally obtain
\[
p_{2k+1}'(u,v)=C_{a,b,2k+1}+ \frac{2k+1}{2^k}uv^k +O_K(\zeta |(u,v)|^{k+1} + |(u,v)|^{k+2} ).
\]
This finishes the proof of Theorem \ref{linearalgebra}, modulo the proof of the claimed representation \eqref{existence}, which we carry out now.\\

We reformulate this problem by using partial derivatives.
Recall that the gradient $\nabla_{u,v}(f)$ defined in \eqref{grad} is a column vector.
We will show that for every $k \geq 1$ and $j=0,1$ there exists $w'=(w_1',\ldots,w_{2k}',1) \in \R^{2k+1}$ with $|w_i'
| \lesssim_{K} \zeta$  such that 
\beq\label{beforelemma}
\begin{pmatrix}[c|c|c|c|c]
\nabla_{u,v} (p_{1} )|_{(a,b)}  & \nabla_{u,v} (p_{2} )|_{(a,b)} & \cdots & \nabla_{u,v} (p_{2k} )|_{(a,b)} & \nabla_{u,v} (p_{2k+1+j} )|_{(a,b)} \\[3pt]
\nabla_{u,v}^2 (p_{1} )|_{(a,b)} & \nabla_{u,v}^2 (p_{2} )|_{(a,b)} & \cdots & \nabla_{u,v}^2 (p_{2k} )|_{(a,b)} & \nabla_{u,v}^2 (p_{2k+1+j} )|_{(a,b)} \\
\vdots & \vdots & \ddots & \vdots & \vdots \\
\nabla_{u,v}^k (p_{1} )|_{(a,b)} & \nabla_{u,v}^k (p_{2} )|_{(a,b)} & \cdots & \nabla_{u,v}^k (p_{2k} )|_{(a,b)} & \nabla_{u,v}^k (p_{2k+1+j} )|_{(a,b)}
\end{pmatrix} w'^\intercal  =: Mw'^\intercal =0.
\endeq
We rewrite the matrix $M'$ by
\[
M' = \begin{pmatrix}
\nabla_{u,v} & \nabla^2_{u,v} & \cdots & \nabla^k_{u,v}
\end{pmatrix}^\intercal
\begin{pmatrix}
p_{1} & p_{2} & \cdots & p_{2k}& p_{2k+1+j}
\end{pmatrix}|_{(a,b)}.
\]
By applying Lemma \ref{changeofvariables} and  multiplying the matrix $A_{a,b}^{-1}$ on both sides of \eqref{beforelemma}, it suffices to show that there exists $w'=(w_1',\ldots,w_{2k}',1) \in \R^{2k+1}$ with $|w_i'| \lesssim_{K} \zeta$ such that
	\[
\begin{pmatrix}
\nabla_{t,s} & \nabla^2_{t,s} & \cdots & \nabla^k_{t,s}
\end{pmatrix}^\intercal \begin{pmatrix}
p_{1} & p_{2} & \cdots & p_{2k} & p_{2k+1+j}
\end{pmatrix}|_{(\alpha,\beta)} w'^\intercal =: Pw'^\intercal= 0.
\]
Since $p_i(t,s)=t^i+s^i$, by direct computations, we obtain that
\begin{displaymath}
\begin{pmatrix}
\nabla_{t,s} & \nabla^2_{t,s} & \cdots & \nabla^k_{t,s}
\end{pmatrix}^\intercal
\begin{pmatrix}
p_{1} & p_{2} & \cdots & p_{2k} & p_{2k+1+j}
\end{pmatrix}
\end{displaymath}
is equal to 
\begin{displaymath}
\begin{split}
\large\begin{pmatrix}
	\frac{\partial p_1}{\partial t} & \frac{\partial p_2}{\partial t} & \cdots &\frac{\partial p_{2k+1+j}}{\partial t} \\[0.50em]
	\frac{\partial p_1}{\partial s} & \frac{\partial p_2}{\partial s} & \cdots &\frac{\partial p_{2k+1+j}}{\partial s} \\[0.50em]
	\frac{\partial^2 p_1}{\partial t^2} & \frac{\partial^2 p_2}{\partial t^2}  & \cdots & \frac{\partial^2 p_{2k+1+j}}{\partial t^2} \\[0.50em]
	0 & 0 &\cdots &0 \\[0.50em]
	\frac{\partial^2 p_1}{\partial s^2} & \frac{\partial^2 p_2}{\partial s^2} &\cdots &\frac{\partial^2 p_{2k+1+j}}{\partial s^2} \\[0.50em]
	\vdots & \vdots &\ddots &\vdots \\[0.50em]
	\frac{\partial^k p_1}{\partial t^k} & \frac{\partial^k p_2}{\partial t^k}  &\cdots & \frac{\partial^k p_{2k+1+j}}{\partial t^k} \\[0.50em]	
	0 & 0 &\cdots & 0 \\[0.50em]	
		\vdots & \vdots &\ddots & \vdots \\[0.50em]
	0 & 0  &\cdots & 0 \\[0.50em]	
	\frac{\partial p_1}{\partial s^k} & \frac{\partial p_2}{\partial s^k} & \cdots & \frac{\partial p_{2k+1+j}}{\partial s^k} \\[0.50em]
\end{pmatrix}.
\end{split}
\end{displaymath}
To simplify the notation, we reorder the rows by applying a linear transformation, and we may assume that $P$ is the matrix defined by
	\begin{gather*}
		\large\begin{pmatrix}
			\frac{\partial p_1}{\partial t}(\alpha,\beta) & \frac{\partial p_2}{\partial t}(\alpha,\beta) & \cdots &\frac{\partial p_{2k+1+j}}{\partial t}(\alpha,\beta) \\[0.50em]
			\frac{\partial^2 p_1}{\partial t^2}(\alpha,\beta) & \frac{\partial^2 p_2}{\partial t^2}(\alpha,\beta)  & \cdots & \frac{\partial^2 p_{2k+1+j}}{\partial t^2}(\alpha,\beta) \\[0.50em]
			\vdots & \vdots &\ddots &\vdots \\[0.50em]
			\frac{\partial^k p_1}{\partial t^k}(\alpha,\beta) & \frac{\partial^k p_2}{\partial t^k}(\alpha,\beta)  &\cdots & \frac{\partial^k p_{2k+1+j}}{\partial t^k}(\alpha,\beta) \\[0.50em]
			\frac{\partial p_1}{\partial s}(\alpha,\beta) & \frac{\partial p_2}{\partial s}(\alpha,\beta) & \cdots &\frac{\partial p_{2k+1+j}}{\partial s}(\alpha,\beta) \\[0.50em]
			\frac{\partial^2 p_1}{\partial s^2}(\alpha,\beta) & \frac{\partial^2 p_2}{\partial s^2}(\alpha,\beta) &\cdots &\frac{\partial^2 p_{2k+1+j}}{\partial s^2}(\alpha,\beta) \\[0.50em]
			\vdots & \vdots &\ddots &\vdots \\[0.50em]
			\frac{\partial^k p_1}{\partial s^k}(\alpha,\beta) & \frac{\partial^k p_2}{\partial s^k}(\alpha,\beta)  &\cdots & \frac{\partial^k p_{2k+1+j}}{\partial s^k}(\alpha,\beta) \\[0.50em]	
			0 & 0 &\cdots & 0 \\[0.50em]	
			\vdots & \vdots &\ddots & \vdots \\[0.50em]
			0 & 0  &\cdots & 0 \\[0.50em]	
		\end{pmatrix}.
	\end{gather*}
	We rewrite the matrix $P$ by
	\[
	\begin{split}
	P&=\begin{pmatrix}
		\frac{\partial }{\partial t} & \cdots & \frac{\partial^k }{\partial t^k} & \frac{\partial }{\partial s} & \cdots & \frac{\partial^k }{\partial s^k} & 0 & \cdots & 0
	\end{pmatrix}^\intercal  \begin{pmatrix}
	p_{1} & p_{2} & \cdots & p_{2k} & p_{2k+1+j}
\end{pmatrix}|_{(\alpha,\beta)}
\\&=
\begin{pmatrix}[c|c|c|c|c|c|c|c|c|c|c]
\gamma^{(1)}(\alpha) & \gamma^{(2)}(\alpha) & \cdots &  \gamma^{(k)}(\alpha) & \gamma^{(1)}(\beta) & 
\gamma^{(2)}(\beta) & \cdots & 
\gamma^{(k)}(\beta) & 0 & \cdots & 0\,
\end{pmatrix}^{\intercal},
	\end{split}
	\]
where $\gamma(t)=(t,t^2,\ldots,t^{2k},t^{2k+1+j})^{\intercal}$, and $\gamma^{(i)}$ indicates the $i$th derivative of $\gamma(t)$.\\

To proceed, we need to compute the determinant of a submatrix of the matrix $P$. This will rely on a formula of the determinant of the generalized Vandermonde matrix due to Kalman \cite{K84}.

\begin{lem}[\cite{K84}]\label{kalman}
	Let $k \geq 1$.
	Let $\tilde{\gamma}(t) = (t,t^2,t^3,\ldots,t^{2k})^{\intercal}$. Then 
	\[
	\det
	\begin{pmatrix}[c|c|c|c|c|c|c|c]
	\tilde{\gamma}^{(1)}(x) & \frac{1}{2!} \tilde{\gamma}^{(2)}(x) & \cdots & \frac{1}{k!} \tilde{\gamma}^{(k)}(x) & \tilde{\gamma}^{(1)}(y) & \frac{1}{2!}\tilde{\gamma}^{(2)}(y) & \cdots & \frac{1}{k!}\tilde{\gamma}^{(k)}(y)
	\end{pmatrix} = (x-y)^{k^2}.
	\]
	Here, $\tilde{\gamma}^{(i)}$ indicates the $i$th derivative of $\tilde{\gamma}(t)$.
\end{lem}
%
%
%
%
%

We first consider the case $\zeta=1$. In this case, it suffices to show that the determinant of the upper left $2k \times 2k$ matrix is non-vanishing whenever $\alpha \neq \beta$, which follows immediately from Lemma \ref{kalman}. Thus, there exists $w'=(w_1',\ldots,w_{2k}',1) \in \R^{2k+1}$ with $|w_i'| \lesssim_K 1$ such that
$Pw'^{T}=0$.

Next we consider the general case $0<\zeta<1$. We write $(\alpha,\beta) = (\zeta\bar\alpha,\zeta\bar\beta)$ so that $0\leq \bar\alpha,\bar\beta \leq 1$ and $|\bar\alpha-\bar\beta| \geq K^{-1}$. We apply the result of the case $\zeta=1$ to $(\bar\alpha,\bar\beta)$ and obtain that there exists $\bar{w}=(\bar{w}_1,\ldots,\bar{w}_{2k},1)$ with $|\bar{w}_i| \lesssim_{K} 1$ such that
\[
\begin{pmatrix}[c|c|c|c|c|c|c|c|c|c|c]
\gamma^{(1)}(\bar\alpha) &  \gamma^{(2)}(\bar\alpha) & \cdots &  \gamma^{(k)}(\bar\alpha) & \gamma^{(1)}(\bar\beta) & \gamma^{(2)}(\bar\beta) & \cdots & \gamma^{(k)}(\bar\beta) & 0 & \cdots & 0
\end{pmatrix}^{\intercal} \bar{w}^{\intercal} = 0.
\]
We put $w' = (w_1',\ldots,w_{2k}',1) := (\zeta^{2k+j}\bar{w}_1,\zeta^{2k-1+j}\bar{w}_2,\ldots,\zeta^{1+j}\bar{w}_{2k},1)$. Note that $|w_i'| \lesssim_{K} \zeta$. By the construction, we obtain
\[
\begin{pmatrix}[c|c|c|c|c|c|c|c|c|c|c]
\gamma^{(1)}(\zeta\bar\alpha) &  \gamma^{(2)}(\zeta\bar\alpha) & \cdots &  \gamma^{(k)}(\zeta\bar\alpha) & \gamma^{(1)}(\zeta\bar\beta) & \gamma^{(2)}(\zeta\bar\beta) & \cdots & \gamma^{(k)}(\zeta\bar\beta) & 0 & \cdots & 0
\end{pmatrix}^{\intercal} w'^{\intercal} = 0.
\]
This completes the proof of Theorem \ref{linearalgebra}.
\end{proof}

\section{Proof of Proposition \ref{dectwodim}}
%
%
%
%

In this section we will finish the proof of Proposition \ref{dectwodim}. Here we will see the motivation of restricting both intervals $J_1$ and $J_2$ to the small interval $[0, \delta^{\epsilon}]$. Roughly speaking, when both $J_1$ and $J_2$ are close to the origin, we are able to approximate the relevant manifold $\mc{M}$ (defined in \eqref{twodimmanifold}) by the moment manifold $\mc{M}_0$ (see \eqref{monomial_manifold}). One advantage of working with the manifold $\mc{M}_0$ is that it is translation-invariant.  Moreover, certain sharp decoupling inequalities for such a manifold have already been established in \cite{GZ18}. The proof there relies crucially on the fact that the manifold is a moment manifold and a translation-invariant manifold. In sharp contrast, the manifold $\mc{M}$ is neither a moment manifold nor a translation-invariant manifold. 

Having a small parameter $\delta^{\epsilon}$ as above will create enough room for a bootstrapping argument (see \eqref{PSinductionhypothesis}). Such kind of a bootstrapping argument can be dated back to the work of Pramanik and Seeger \cite{PS07}. \\

Recall that we need to show that
\beq\label{twodimdec}
\big\|E_{L(J_1,J_2)}^{\mathcal{M}}g\big\|_{L^{p}(w_{B^r})}
\leq C_{p,K,\epsilon} \delta^{-2(\frac12-\frac{1}{p})-C\epsilon }
\biggl(\sum_{\substack{\square \subset [0,1]^2\\ l(\square)=\delta} } \big\|E_{\square}^{\mathcal{M}}g\big\|_{L^{p}(w_{B^r})}^{p} \biggr)^{\frac{1}{p}}.
\endeq
Here it is important to keep in mind that the boxes $\Box$ that appear in the right hand side of \eqref{twodimdec} all have non-empty intersections with $L(J_1, J_2)$. To prove \eqref{twodimdec}, we will apply an inductive argument and prove 
\beq\label{PSinductionhypothesis}
\begin{split}
\big\|E_{L(J_1,J_2)}^{\mathcal{M}}g\big\|_{L^{p}(w_{B^r})}
& \leq (C_{p,\epsilon}C_{p,K,\epsilon})^{4m} \delta^{-C\epsilon }\delta^{-{2\frac{m\epsilon}{r}}(\frac12-\frac1p)-{m\epsilon^2}}\\
& 
\biggl(\sum_{\substack{R \cap L(J_1,J_2) \neq \emptyset\\ l(R)=\delta^{\frac{m\epsilon}{r}} }} \big\|E_{R}^{\mathcal{M}}g\big\|_{L^{p}(w_{B^r})}^{p} \biggr)^{\frac{1}{p}}
\end{split}
\endeq
for every integer $m$ with $r \leq m \leq r\epsilon^{-1}$. The desired inequality \eqref{twodimdec} follows from \eqref{PSinductionhypothesis} with $m=r\epsilon^{-1}$.\\

Let us start with proving the base case of \eqref{PSinductionhypothesis}, that is, the case $m=r$. This follows from $L^2$ orthogonality and interpolation with a trivial $L^{\infty}$ bound: 
\beqq
\begin{split}
\big\|E_{L(J_1,J_2)}^{\mathcal{M}}g\big\|_{L^{p}(w_{B^r})}
\leq C_{p,K,\epsilon} \delta^{-C\epsilon }
\biggl(\sum_{R \cap L(J_1,J_2) \neq \emptyset: l(R)=\delta^{\epsilon} } \big\|E_{R}^{\mathcal{M}}g\big\|_{L^{p}(w_{B^r})}^{p} \biggr)^{\frac{1}{p}}.
\end{split}
\endeqq
%
%
%
%
%
Suppose that we have proven \eqref{PSinductionhypothesis} for some $m=m_0<r\epsilon^{-1}$. We will show that \eqref{PSinductionhypothesis} holds true for $m=m_0+1$. By the induction hypothesis, we have
\beq\label{induction_hypo}
\begin{split}
\big\|E_{L(J_1,J_2)}^{\mathcal{M}}g\big\|_{L^{p}(w_{B^r})}
& \leq (C_{p,\epsilon}C_{p,K,\epsilon})^{4m_0} \delta^{-C\epsilon }\delta^{-{2\frac{m_0\epsilon}{r}}(\frac12-\frac1p)-{m_0\epsilon^2}}\\
&\biggl(\sum_{R \cap L(J_1,J_2) \neq \emptyset: l(R)=\delta^{\frac{m_0\epsilon}{r}} } \big\|E_{R}^{\mathcal{M}}g\big\|_{L^{p}(w_{B^r})}^{p} \biggr)^{\frac{1}{p}}.
\end{split}
\endeq
Fix a square $R$ intersecting $L(J_1,J_2)$ with side length $\delta^{m_0\epsilon/r}$. For simplicity, we put $\gamma=\frac{m_0\epsilon}{r}$. We claim that
\beq\label{claim}
\big\|E_{R}^{\mathcal{M}}g\big\|_{L^{p}(w_{B^r})}
\leq (C_{p,\epsilon}C_{p,K,\epsilon})^{4} (\delta^{\frac{\epsilon}{r}})^{-2(\frac12-\frac{1}{p})-\epsilon^2 }
\biggl(\sum_{\substack{R' \subset  R : R' \cap L(J_1,J_2) \neq \emptyset,\\ l(R')=\delta^{\gamma+(\epsilon/r)} }} \big\|E_{R'}^{\mathcal{M}}g\big\|_{L^{p}(w_{B^r})}^{p} \biggr)^{\frac{1}{p}}.
\endeq
This claim, combined with \eqref{induction_hypo} will finish the proof of \eqref{PSinductionhypothesis} with $m=m_0+1$, thus close the induction. \\

It remains to prove \eqref{claim}. Suppose that $R=(a,b)+[0,\delta^{\gamma}]^2$ for some point $(a, b)$. 
Define
$
h(u,v):=g( u+a,v+b).
$
By a change of variables, the claimed estimate \eqref{claim} follows from
\begin{gather}\label{claim2}
\big\|E_{R_1}^{\mathcal{M}_{(a, b)}}h\big\|_{L^{p}(w_{B^r})}
\leq (C_{p,\epsilon}C_{p,K,\epsilon})^3 (\delta^{\frac{\epsilon}{r}})^{-2(\frac12-\frac{1}{p}) -\epsilon^2}
\biggl(\sum_{\substack{R_2 \subset  [0,\delta^{\gamma}]^2 \\ l(R_2)=\delta^{\gamma+(\epsilon/r)} }} \big\|E_{R_2}^{\mathcal{M}_{(a, b)} }h\big\|_{L^{p}(w_{B^r})}^{p} \biggr)^{\frac{1}{p}}.
\end{gather}
Here, $R_1=[0,\delta^{\gamma}]^2$, and the manifold $\mc{M}_{(a,b)}$ is defined at the beginning of Section 3.
According to Theorem 3.1 with $\zeta=\delta^{\epsilon}$, we obtain
\[
\mathcal{M}_{(a, b)} \cong_M \big\{ (u,v,q_3(u,v),\ldots,q_n(u,v) ) \big\}=:\mathcal{M}'_{(a, b)},
\]
where
\begin{displaymath}
q_i(u,v) = \left\{ \begin{array}{ll}
\frac{2k+1}{2^{k}}uv^{k} + O_{K}\big(\delta^{\epsilon}|(u,v)|^{k+1}+ |(u,v)|^{k+2} \big)& \textrm{with $k=\frac{i-1}{2}$ when $i$ is odd}\\[3pt]
\frac{1}{2^{k}}v^{k+1}+ O_{K}\big(\delta^{\epsilon}|(u,v)|^{k+1} + |(u,v)|^{k+2} \big)& \textrm{with $k=\frac{i-2}{2}$ when $i$ is even}.
\end{array} \right.
\end{displaymath}
We now apply the change of variables 
$
(u,v) \mapsto \delta^{\gamma}(u,v)$. Denote  $\tilde{h}(u,v):=h(\delta^{\gamma}u,\delta^{\gamma}v)$ and  $R_3:=[0,1]^2$. 
By a simple scaling argument, \eqref{claim2} follows from
\beq\label{iteration}
\begin{split}
\big\|E_{R_3}^{\mathcal{M}'}\tilde{h}\big\|_{L^{p}(w_{B^{r-r\gamma} })}
\leq
(C_{p,\epsilon}C_{p,K,\epsilon})^2 (\delta^{\frac{\epsilon}{r}})^{-2(\frac12-\frac{1}{p}) -\epsilon^2}
\biggl(\sum_{R_4 \subset  [0,1]^2 : l(R_4)=\delta^{\epsilon/r} } \big\|E_{ R_4 }^{\mathcal{M}'}\tilde{h}\big\|_{L^{p}(w_{B^{r-r\gamma} })}^{p} \biggr)^{\frac{1}{p}}.
\end{split}
\endeq
Here,
\[
\mathcal{M}'= \big\{ (u,v,\bar{q}_3(u,v),\ldots,\bar{q}_n(u,v) ): (u,v) \in [0,1]^2 \big\},
\]
with
\begin{displaymath}
\bar{q}_i(u,v) = \left\{ \begin{array}{ll}
\frac{2k+1}{2^{k}}uv^{k} + O_{K}\big(\delta^{\epsilon}|(u,v)|^{k+1}\big)& \textrm{with $k=\frac{i-1}{2}$ when $i$ is odd}\\[3pt]
\frac{1}{2^{k}}v^{k+1}+ O_{K}\big(\delta^{\epsilon}|(u,v)|^{k+1} \big)& \textrm{with $k=\frac{i-2}{2}$ when $i$ is even}.
\end{array} \right.
\end{displaymath}
Define a new manifold 
\beq\label{monomial_manifold}
\mathcal{M}_0 := \big\{ (u,v,Q_3(u,v),\ldots,Q_n(u,v)): (u,v)\in[0,1]^2 \big\},
\endeq
where
\begin{displaymath}
Q_i(u,v) := \left\{ \begin{array}{ll}
\frac{2k+1}{2^{k}}uv^{k} & \textrm{with $k=\frac{i-1}{2}$ when $i$ is odd}\\[3pt]
\frac{1}{2^{k}}v^{k+1} & \textrm{with $k=\frac{i-2}{2}$ when $i$ is even}.
\end{array} \right.
\end{displaymath}
It is straightforward to see that the distance between $\mc{M}'$ and $\mc{M}$ is $O_{K}(\delta^{\epsilon})$. By the uncertainty principle, the errors $O_{K}(\delta^{\epsilon})$ are negligible on the ball $B_{\delta^{-\epsilon}}$. Moreover, $r-r\gamma\ge \epsilon$. 
Therefore, to prove \eqref{iteration}, it suffices to prove the same estimate with $\mc{M}_0$ in place of $\mc{M}'$ in it. However, such an estimate has already been established in \cite{GZ18} (see Theorem 1.1 and Example 1.4 therein). This concludes the proof of \eqref{claim2}. 


\end{document}